\documentclass[10pt,dvips]{article}
\usepackage[a5paper,layoutwidth=148.5mm,layoutheight=7.77817in,margin=36pt,footskip=19pt]{geometry}
\usepackage{amssymb,amstext,amsmath}

\usepackage{natbib}
\bibpunct{(}{)}{;}{a}{}{,}
\newcommand\urlprefix{}
\usepackage[utf8]{inputenc}
\usepackage[T1]{fontenc}
\usepackage[french,german,english]{babel}

\usepackage{amsthm}
\usepackage{etoolbox}
\makeatletter\def\th@plain{\slshape}\makeatother
\makeatletter\patchcmd{\th@remark}{\itshape}{\slshape}{}{}\makeatother

\AtBeginDocument{}

\newtheorem{theorem}{Theorem}[section]
\newtheorem{corollary}[theorem]{Corollary}
\newtheorem{lemma}[theorem]{Lemma}
\newtheorem{proposition}[theorem]{Proposition}

\theoremstyle{definition}

\theoremstyle{remark}

\newtheorem{example}[theorem]{Example}
\newtheorem{remark}[theorem]{Remark}

\usepackage{stmaryrd}
\usepackage[shortlabels]{enumitem}
\setlist{noitemsep}

\newcommand{\tS}{S_{\rM}}
\newcommand{\rM}{\mathrm{M}}
\newcommand\lrb[1] {\llbracket #1 \rrbracket}
\usepackage{mathtools}
\mathtoolsset{mathic}
\PassOptionsToPackage{hyphens}{url}
\PassOptionsToPackage{hidelinks}{hyperref}
\usepackage{hyperref}
\usepackage{breakurl}
\usepackage{doi}
\ifx\texorpdfstring\undefined\newcommand\texorpdfstring[2]{#1}\fi
\begin{document}

\title{
  Regular entailment relations}

\author{Thierry Coquand \and Henri Lombardi \and Stefan Neuwirth}


%
\maketitle

\begin{abstract} 

Inspired by the work of Lorenzen on the theory of preordered groups in the forties and fifties, 
we define regular entailment relations and show a crucial theorem for this 
structure. We also describe equivariant systems of ideals à la Lorenzen and show 
that the remarkable regularisation process invented by him yields a regular 
entailment relation. By providing constructive objects and arguments, we pursue Lorenzen's aim of ``bringing to light the basic, pure concepts in 
their simple and transparent clarity''.\smallskip

\smallskip


\end{abstract}

\section*{Introduction}

Paul Lorenzen carried out, in a series of four articles, an analysis of multiplicative ideal theory in terms of embeddings into an $l$-group. In \citealt{Lor1939}, he formulated the problem in the language of semigroups instead of integral domains. The endeavour of \citealt{Lor1950} was to remove the condition of commutativity; the unavailability of the Grothendieck group construction led him to discover the ``regularity condition'' and to propose a far-reaching reformulation of embeddability into a product of linearly preordered groups in terms of ``regularisation''.  He also arrived at the formulation of the concept of equivariant system of ideals, as below, and of entailment relation. The article \citealt{Lor1952} broadened his analysis to the more general case  of a monoid acting on a preordered set. Our research started as a study of \citealt{Lor1953}, in which he proved a result that suggested Theorem~\ref{thmain} to us.

 If $(G,0,+,-,\leqslant)$ is a preordered commutative group and we have a morphism $f:G\rightarrow L$ with $L$ an $l$-group, then we can define a relation $A\vdash B$ between \emph{nonempty} finite subsets of $G$
by $\bigwedge f(A)\leqslant_L \bigvee f(B)$. This relation satisfies the following conditions.
\begin{itemize}
\item [$(R_1)$] $A\vdash B$ if $A\supseteq A'$ and $B\supseteq B'$ and $A'\vdash B'$\hfill(weakening);
\item [$(R_2)$] $A\vdash B$ if $A,x\vdash B$ and $A\vdash B,x$\hfill(cut);
\item [$(R_3)$] $a\vdash b$ if $a\leqslant b$ in $G$;
\item [$(R_4)$] $A\vdash B$ if $A+x\vdash B+x$\hfill(translation);
\item [$(R_5)$] $a+x,b+y\vdash a+b,x+y$\hfill(regularity).
\end{itemize}
We make the following
abuses of notation for finite sets: we write $a$ for the
singleton consisting of~$a$, and $A,A'$ for the union of the sets~$A$
and~$A'$; note that our framework requires only a naive set theory.
 We call any relation which satisfies
these conditions a \emph{regular entailment relation} for the preordered group~$G$. The remarkable last condition is called the \emph{regularity condition}.

Note that the converse of a regular entailment relation for $(G,0,+,-,\leqslant)$ is a regular entailment relation for $(G,0,+,-,\geqslant)$ (the group with the converse preorder). When we use this, we say that a result follows from another one \emph{symmetrically}.

 Any relation satisfying the first three conditions defines in a canonical way a(n unbounded)
distributive lattice $L$ with a natural monotone map $G\to L$: see \citealt[Satz~7]{Lor1951}; \citealt[Theorem~1 (obtained independently)]{CeCo2000}. 

The goal of this note is essentially to show that this distributive lattice has
a (canonical) $l$-group structure, simplifying some arguments in \citealt{Lor1953}.
This is done in Theorem~\ref{thmain}. In Section 2, we explain how to define a regular entailment relation through a predicate on nonempty finite subsets of $G$.
In Section 3, we define ``equivariant systems of ideals'' à la Lorenzen and we show how to express this notion through a predicate on nonempty finite subsets of $G$. 
In Section 4, we explain how Lorenzen ``regularises'' an equivariant system of ideals, which leads to the Lorenzen group of this system of ideals (Theorem~\ref{thLorGroup}). In Section 5, we explain the link with a constructive version of the Lorenzen-Clifford-Dieudonné theorem.
In Section 6, we explain the link with the Prüfer way of defining the Lorenzen group of a system of ideals.
In Section~7, we give a constructive version of a remarkable theorem of Lorenzen
which uses the regularity condition in the noncommutative case. Finally, in Section 8,
we give examples illustrating some constructions described in the paper.

\smallskip 
The results of this research complement the ones of  \cite{CLN2018}: we introduce
various equivalent presentations of regular entailment relations and we also provide
a noncommutative version and several examples.

\section{General properties of regular entailment relations}
 
 A first consequence of regularity is the following.

\begin{proposition}\label{main1}
We have \(a,b\vdash a+x,b-x\) and \(a+x,b-x\vdash a,b\). In particular,
\(a\vdash a+x,a-x\) and \(a+x,a-x\vdash a\).
\end{proposition}

\begin{proof}
By regularity, we have $x + (a-x),(b - 2x) + 2x \vdash 
 x+(b-2x),(a-x)+2x$, which is $a,b\vdash a+x,b-x$. The other
claim follows symmetrically.
\end{proof}

\begin{corollary}
  In the distributive lattice \(L\) defined by the (unbounded)
  entailment relation~\(\vdash\),
  \(\bigwedge A\leqslant (\bigwedge(A+x))\vee (\bigwedge(A-x))\).
\end{corollary}

\begin{proof}
  In~$L$, we have
  \({(\bigwedge_{a\in A}(a+x))}\vee {(\bigwedge_{b\in A} (b-x))}={\bigwedge_{a,b\in A}((a+x)\vee(b-x))}\), so that
  this follows from Proposition~\ref{main1}.
\end{proof}

\begin{corollary}\label{R1}
If we have \(A,A+x\vdash B\) and \(A,A-x\vdash B\), then \(A\vdash B\). Symmetrically, if \(A\vdash B,B+x\) and
\(A\vdash B,B-x\), then \(A\vdash B\).
\end{corollary}

\begin{lemma}\label{R2}
We have \(A,A+x\vdash B\) iff \(A\vdash B,B-x\).
\end{lemma}

\begin{proof}
We assume $A,A+x\vdash B$ and we prove $A\vdash B,B-x$. 
By Corollary~\ref{R1}, it is enough
to show $A,A-x\vdash B,B-x$, but this follows from $A,A+x\vdash B$ by translating by $-x$
and then weakening. The other direction is symmetric.
\end{proof}

\begin{lemma}\label{conv}
If \(0\leqslant p\leqslant q\), then \(a,a+qx \vdash a+px\).
\end{lemma}

\begin{proof}
We prove this by induction on $q$. It holds for $q = 0$ and $q = 1$. If it holds for $q\geqslant1$,
we note that we have $a,a+(q+1)x \vdash a + x,a+qx$ by regularity, and since
$a,a+qx\vdash a+x$ by induction hypothesis, we get $a,a+(q+1)x\vdash a+x$ by cut.
By induction hypothesis, we have $a,a+qx\vdash a+px$ for $p\leqslant q$, and hence
$a+x,a+(q+1)x\vdash a+(p+1)x$. By cut with $a,a+(q+1)x\vdash a+x$ 
we get $a,a+(q+1)x\vdash a+(p+1)x$.
\end{proof}

 Given a regular entailment relation $\vdash$ and an element $x$, we now describe the 
regular entailment relation $\vdash_x$ for which we force $0\vdash_x x$. This relation
exists by universal algebra, but let us define
that $A\vdash_x B$ holds iff there exists $p$ such that $A,A+px\vdash B$, iff (by Lemma~\ref{R2}) there
exists $p$ such that $A\vdash B,B-px$. We are going to show that this is the least regular
entailment relation containing $\vdash$ and such that  $0\vdash_x x$. 
We have $0\vdash_x x$ since $0,x\vdash x$.

 Note that, by using Lemma~\ref{conv}, if we have $A,A+px\vdash B$, we also have
$A,A+qx\vdash B$ for $q\geqslant p$.

\begin{proposition}\label{main2}
The relation \(\vdash_x\) is a regular entailment relation. It is the least regular entailment
relation containing \(\vdash\) and such that \(0\vdash_x x\).
\end{proposition}

\begin{proof}
The only complex condition is the cut rule. We assume $A,A+px\vdash B,u$ and
$A,A+qx,u,u+qx\vdash B$, and we prove $A\vdash_x B$. By Lemma~\ref{conv}, we can assume $p=q$.
We write $y = px$ and we have $A,A+y\vdash B,u$ and $A,A+y,u,u+y\vdash B$. 
We write $C = A,A+y,A+2y$ and we prove $C\vdash B$.

 We have, by weakening, $C\vdash B,u$ and $C,u,u+y\vdash B$ and $C\vdash B+y,u+y$.
By cut, we get $C,u\vdash B,B+y$. By Lemma~\ref{R2}, this is equivalent to $C,u,C-y,u-y\vdash B$.
We also have $C,u,C+y,u+y\vdash B$ by weakening $C,u,u+y\vdash B$. Hence by Lemma~\ref{R1}
we get $C,u\vdash B$. Since we also have $C\vdash B,u$, we get $C\vdash B$ by cut.

By Lemma~\ref{conv} we have $A,A+2y\vdash B$, which shows $A\vdash_x B$.
\end{proof}

\begin{proposition}\label{main3}
If \(A\vdash_x B\) and \(A\vdash_{-x} B\), then \(A\vdash B\).
\end{proposition}

\begin{proof}
We have $A,A+px\vdash B$ and $A,A-qx\vdash B$. Using Lemma~\ref{conv} we can assume $p=q$
and then conclude by Corollary~\ref{R1}.
\end{proof}

 Proposition~\ref{main3} implies that in order to prove an entailment involving certain elements, we can
always assume that all elements occurring in the proof are 
linearly preordered for the relation $a \vdash b$.
This corresponds to the informal covering principle by quotients for $l$-groups \cite[Principle XI-2.10]{CACM}. Here are two direct applications.

\begin{proposition}\label{key}
We have \(A\vdash b_1,\dots,b_m\) iff \(A-b_1,\dots,A-b_m\vdash 0\).
\end{proposition}

 Thus $A\vdash B$ iff $A-B\vdash 0$ iff $0\vdash B-A$. The first equivalence is
exactly Proposition~\ref{key}, and the second equivalence follows symmetrically.

\begin{proposition}\label{cancel}
If $A+b_1,\dots,A+b_m\vdash b_j$ for $j = 1,\dots,m$, then $A\vdash 0$.
\end{proposition}

 It follows from Proposition~\ref{cancel} that if we consider the monoid of formal elements
$\bigwedge A$ with the operation $\bigwedge A+\bigwedge B = \bigwedge (A+B)$, preordered by the relation
$\bigwedge A\leqslant \bigwedge B$ iff $A\vdash b$ for all $b$ in $B$, we get a \emph{cancellative}
monoid.

The \emph{Grothendieck $l$-group} of a meet-monoid $(M,+,0,\wedge)$ is the $l$-group that it freely generates. Its group structure is given by the Grothendieck group of the monoid $(M,+,0)$.

\begin{corollary}
The distributive lattice defined by the Grothendieck \(l\)-group 
of the previously defined cancellative monoid coincides with the distributive lattice defined by the relation \(\vdash\).
\end{corollary}

 We have realised in this way our goal.

 \begin{theorem} \label{thmain}
The distributive lattice \(V\) generated by a regular
entailment relation has a canonical \(l\)-group structure 
for which the natural preorder morphism \(\varphi:G\to V\)  
is a group morphism. 
 \end{theorem}

Note that we may have $a\vdash b$ without $a\leqslant b$, so $\varphi$
is not necessarily injective. 

 Here is another consequence of the fact 
that we can
always assume that  elements are 
linearly preordered 
for the relation $a \vdash b$.

\begin{corollary}
If \(a_1+\dots +a_n = 0\) then \(a_1,\dots,a_n\vdash 0\).
\end{corollary}

\begin{corollary}
If \(a_1+\dots +a_n = b_1+\dots+b_n\) then \(a_1,\dots,a_n\vdash b_1,\dots,b_n\).
\end{corollary}

\begin{proof}
We have $\Sigma_{i,j} (a_i-b_j) = 0$ and we can apply the previous result and Proposition~\ref{key}.
\end{proof}

\section{Another presentation of regular entailment relations}

It follows from Proposition~\ref{key} that the relation $\vdash$ is completely determined
by the predicate $A\vdash 0$ on nonempty finite subsets of the group. 
Let us analyse the properties satisfied by this predicate $R(A) = A\vdash 0$.
Firstly, it satisfies
\begin{itemize}
\item [$(P_3)$] $R(a)$ if $a\leqslant 0$ in $G$.
\end{itemize}

 Secondly, it is \emph{monotone}:
 \begin{itemize}
 \item [$(P_1)$] $R(A)$ if $R(A')$ and $A'\subseteq A$\hfill(weakening).
\end{itemize}

 The cut rule can be stated as
$R(A-B)$ if $R(A-B,x-B)$ and $R(A-B,A-x)$, so we get the following property, since we can
assume $x=0$ by translating and replace $B$ by $-B$:
\begin{itemize}
\item [$(P_2)$] $R(A+B)$ if $R(A+B,A)$ and $R(A+B,B)$\hfill(cut).
\end{itemize}

 Finally, the regularity condition gives $R(a-b,b-a,x-y,y-x)$, which simplifies, 
using $(P_1)$, into
\begin{itemize}
\item [$(P_5)$] $R(x,-x)$\hfill(regularity).
\end{itemize}

 We get in this way another presentation of a regular entailment relation as a predicate satisfying the
conditions $(P_1),(P_2),(P_3),(P_5)$: if $R$ satisfies these properties and $A\vdash B$ is defined by $R(A-B)$,
then we get a regular entailment relation (we have one axiom less since the translation property
``$A\vdash B$ if $A+x\vdash B+x$\kern1pt'' is automatically satisfied).




\section{Equivariant  systems of ideals}

 Let us make the same analysis for the notion of  \emph{equivariant system of ideals}. 
A \emph{system of ideals for a preordered set $G$} can be defined à la Lorenzen as a single-conclusion entailment relation, i.e.\ a 
relation $A\rhd x$ between nonempty finite subsets~$A$ of~$G$ and elements~$x$ of~$G$ satisfying the following conditions.
\begin{enumerate}
\item [$(S_1)$] $A\rhd x$ if $A\supseteq A'$ and $A'\rhd x$\hfill(weakening);
\item [$(S_2)$] $A\rhd x$ if $A,y\rhd x$ and $A\rhd y$\hfill(cut);
\item [$(S_3)$] $a\rhd x$ if $a\leqslant x$ in $G$.
\end{enumerate}
A system of ideals for a preordered group $G$ is said to be \emph{equivariant} when it satisfies the condition
\begin{enumerate}
\item [$(S_4)$] $A\rhd x$ if $A+y\rhd x+y$\hfill(translation).
\end{enumerate}

When we have an equivariant system of ideals, let us consider the predicate $S(A) = A\rhd 0$. This predicate satisfies the following conditions.
\begin{enumerate}
\item [$(P_1)$] $S(A)$  if $A\supseteq A'$ and $S(A')$;
\item [$(P'_2)$] $S(A)$  if $S(A,u)$ and $S(A-u)$\hfill(cut);
\item [$(P_3)$] $S(a)$ if $a\leqslant 0$ in $G$.
\end{enumerate}

 Conversely, if $S$ satisfies $(P_1)$, $(P'_2)$ and $(P_3)$ and if we define
$A\rhd x$ by $S(A-x)$, then $\rhd$ is an equivariant system of ideals, so that $S$~is just another presentation for it. 

To an equivariant system of ideals $S$ we can clearly associate the relation \hbox{$A\leqslant_S B$} given by
``$A\rhd b$ for all $b$ in $B$\,'', and we define thus a preordered monoid
with $A+B$ as monoid operation and $A\wedge B= A, B$ as meet operation.
We call the corresponding preordered monoid \emph{the meet-monoid generated by\/ $S$ on\/ $G$}.  

Conversely, consider for a preordered group $(G,\leqslant)$ any preorder $\leq$ on the monoid of finite nonempty subsets with $a\leqslant b \Rightarrow a\leq b$, the meet operation $A\wedge B$ defined as $A,B$ and the monoid operation $A+B$.
Then  we get
the equivariant system of ideals $A\rhd b =  A\leq b$.

\section{Regularisation of an equivariant system of ideals}

 Note that both notions, reformulations of regular entailment relation and of equivariant system of ideals,
are now predicates on nonempty finite subsets of~$G$. 
We say that an equivariant system of ideals is \emph{regular} if it satisfies $(P_2)$ and $(P_5)$.

 The following proposition follows from Proposition~\ref{cancel}.

\begin{proposition}\label{Prufer}
Let \(S\) be an equivariant system of ideals for a preordered group \(G\). Then 
the meet-monoid generated by \(S\) on \(G\) is cancellative if, and only if, \(S\) is regular.
\end{proposition}

\begin{proof}
If $S$ is regular, then $\leqslant_S$ is cancellative by Proposition~\ref{cancel}.
Conversely, if $\leqslant_S$ is cancellative, then the meet-monoid it defines
embeds into its \foreignlanguage{french}{Grothendieck} $l$-group, which is a distributive lattice.
\end{proof}

 We always have the \emph{least} equivariant  system of ideals for a preordered group~$G$: 
$\tS(A)=A\rhd_\rM 0$ iff $A$ contains an element
$\leqslant 0$ in $G$. It clearly satisfies $(P_1)$ and $(P_3)$, and it satisfies $(P'_2)$:
if $A,u\rhd_\rM0$ then either $A\rhd_\rM0$ or $u\rhd_\rM 0$, and if $u\rhd_\rM 0$ then $A\rhd_\rM u$ implies $A\rhd_\rM0$.

\smallskip Note also that equivariant systems of ideals are closed under arbitrary intersections and directed unions.

\smallskip 
Let $S$ be an equivariant system of ideals.
We define $T_x(S)$ to be the least equivariant system of ideals $Q$ containing $S$ and such that $Q(x)$.
We have $T_xT_y = T_yT_x$ and $T_x(S\cap S') = T_x(S)\cap T_x(S')$ directly from this definition.
\citet[page~516]{Lor1950} found an elegant direct description of $T_x(S)$.

\begin{proposition}\label{LorenzenTrick}
\(T_x(S)(A)\) iff there exists \(k\geqslant 0\) such that \(S(A,A-x, \dots,\allowbreak A-kx)\).
\end{proposition}

\begin{proof}
If we have $A,A-x,\dots,A-kx\leqslant_S u$ and
$A,A-x,\dots,A-lx,u,{u-x},\dots,\allowbreak u-lx\leqslant_S v$, then we have, by $l$ cuts,
$A,A-x,\dots,A-(k+l)x\leqslant_S v$.
\end{proof}

\begin{remark}\label{rem}
Note that, in contradistinction with Lemma~\ref{conv}, we cannot simplify this condition to $S(A,A-kx)$ in general: see Examples~\ref{exa1} and~\ref{exa2}.
\end{remark}

\smallskip 
 We next define $U_x(S) = T_x(S)\cap T_{-x}(S)$. We have $U_xU_y = U_yU_x$.

\begin{lemma}
If \(S\) is an equivariant system of ideals such that \(U_x(S) = S\) for all~\(x\), then \(S\) is regular.
\end{lemma}

\begin{proof}

 We show that conditions $(P_5)$ and $(P_2)$ hold.

\smallskip 

 We have $S(x,-x)$ since we have both $T_x(S)(x,-x)$ and $T_{-x}(S)(x,-x)$. This shows $(P_5)$.

\smallskip 

 Let us show $(P_2)$. We assume $\bigwedge (A+B)\wedge \bigwedge B\leqslant_S 0$
 and $\bigwedge (A+B)\wedge \bigwedge A\leqslant_S 0$, 
and we show $\bigwedge (A+B)\leqslant_S 0$. 

Note that we have $T_{a}(S)(A+B)$ for any $a$ in $A$ 
by monotonicity: forcing $a\leqslant_S 0$, we have ${\bigwedge (A+B)}\leqslant_{T_a(S)} \bigwedge B$,
and so $\bigwedge (A+B)\leqslant_{T_a(S)} 0$ follows from  $\bigwedge (A+B)\wedge \bigwedge B\leqslant_{T_a(S)} 0$.

Let $T$ be the composition of all the $T_{-a}$ with $a$ in~$A$: we force $0\leqslant_S a$ for all $a$
in $A$.  We have $\bigwedge B\leqslant_{T(S)} \bigwedge (A+B)$, and so 
$\bigwedge B\leqslant_{T(S)} 0$ follows from  $\bigwedge (A+B)\wedge \bigwedge B\leqslant_{T(S)} 0$.
This implies $\bigwedge (A+B)\leqslant_{T(S)} \bigwedge A$, and so $\bigwedge (A+B)\leqslant_{T(S)} 0$
follows from $\bigwedge (A+B)\wedge \bigwedge A\leqslant_{T(S)} 0$.

Together, these two facts prove, for the composition~$U$ of all the~$U_a$ with $a$ in~$A$, that
$\bigwedge(A+B)\leqslant_{U(S)}0$. Since $U(S)=S$, we get $\bigwedge (A+B)\leqslant_S 0$, as desired.
\end{proof}

Let us define $L(S)$ as the (directed) union of the $U_{x_1}\cdots U_{x_n}(S)$,
as \citet[\S 2 and p.\ 23]{Lor1953} did. We get the following theorem.

\begin{theorem} \label{thLorGroup}
\(L(S)\) is the least regular system containing \(S\); in other
words, it is the \emph{regularisation} of \(S\). The \(l\)-group granted by Theorem~\ref{thmain} for \(L(S)\) is called the \emph{Lorenzen \(l\)-group} associated to the equivariant system of ideals \(S\). 
\end{theorem}

\section{A constructive version of the Lorenzen-Clifford-\foreignlanguage{french}{Dieudonné} Theorem}

 In particular, we can start from the least equivariant system of ideals
 for a given preordered group $G$.
In this case, we have $L(\tS)(A)$ 
iff there exist $x_1,\dots,x_n$ such that for any choice $\epsilon_1,\dots,\epsilon_n$ of signs $\pm1$
we can find $k_1,\dots,k_n\geqslant 0$ and $a$ in $A$ such that $a+\epsilon_1 k_1 x_1+\cdots+\epsilon_n k_n x_n \leqslant 0$.
We clearly have by elimination: if $L(\tS)(a)$, then $na\leqslant 0$ for some $n>0$. We can
then deduce from this a constructive version of the Lorenzen-Clifford-Dieudonné Theorem.

\begin{theorem}
For any commutative preordered group \(G\), we can build an \(l\)-group \(L\) and a map \(f:G\rightarrow L\)
such that \(f(a)\geqslant 0\) iff there exists \(n>0\) such that \(na\geqslant 0\).
More generally, we have \(f(a_1)\vee\dots \vee f(a_k)\geqslant 0\) iff there
exist \(n_1,\dots,n_k\geqslant 0\) such that \(n_1a_1+\cdots+n_ka_k\geqslant 0\) and
\(n_1+\cdots+n_k>0\).
\end{theorem}

Note that this $l$-group $L$ is the $l$-group freely generated by the preordered group~$G$.

\section{Prüfer's definition of the regularisation}

\citet{Pru1932} found the following direct definition of the regularisation, which follows
 at once from Proposition~\ref{Prufer}. 

\begin{theorem}\label{thPruferRegularisation}
The regularisation \(R\) of an equivariant system of ideals \(S\) can be defined by \(R(A)\) holding iff there exists \(B\) such that \(A+B\leqslant_S B\).
\end{theorem}

This gives another proof that if we have $L(\tS)(a)$ then $na\leqslant 0$ for some $n>0$:
if we have $B$ such that $a+ B\leqslant_{\tS}  B$ then we can find
a cycle $a+b_2\leqslant b_1$, \dots, $a+b_1\leqslant b_n $ with $b_1,\dots,b_n\in B$, and then $na\leqslant 0$.

\section{The \texorpdfstring{\textit{l\/}}{l}-group structure in the noncommutative case}

 If $G$ is a not necessarily commutative preordered group, we use a multiplicative
notation and we define a \emph{regular entailment relation} by the following conditions.
\begin{itemize}
\item [$(R_1)$] $A\vdash B$ if $A\supseteq A'$ and $B\supseteq B'$ and $A'\vdash B'$\hfill(weakening);
\item [$(R_2)$] $A\vdash B$ if $A,x\vdash B$ and $A\vdash B,x$\hfill(cut);
\item [$(R_3)$] $a\vdash b$ if $a\leqslant b$ in $G$;
\item [$(R_4)$] $A\vdash B$ if $xAy\vdash xBy$\hfill(translation);
\item [$(R_5)$] $xa,by\vdash xb,ay$\hfill(regularity).
\end{itemize}

Note that $(R_5)$ is satisfied in linearly preordered groups: if $a\leqslant b$, then $xa\wedge by\leqslant xa\leqslant xb\leqslant xb\vee ay$, and if $b\leqslant a$, then $xa\wedge by\leqslant by\leqslant ay\leqslant xb\vee ay$.

Let $\vdash$ be a regular entailment relation and let $(V,\leqslant_V)$ be
the corresponding distributive lattice; then $(R_4)$ shows that we
have a left and right action of $G$ on~$\leqslant_V$.

We define $\leqslant_{a,b}$ to be the lattice preorder with left and
right action of $G$ on it obtained from~$\leqslant_V$ by forcing
$b\leqslant_{a,b} a$.

 We define $u\leqslant^{a,b} v$ by ``$xa\wedge uy\leqslant_Vxb\vee vy$ for all $x$ and $y$
in $G$\,''.

\begin{lemma}
We have \(xa\wedge by\leqslant_Vxb\vee ay\) for all \(a\) and \(b\) in \(V\) and
all \(x\) and \(y\) in~\(G\).
\end{lemma}

\begin{proof} This holds for $a$ and $b$ in $G$.
Then, if we have 
 $xa_1\wedge by\leqslant_Vxb\vee a_1y$
and 
 $xa_2\wedge by\leqslant_Vxb\vee a_2y$,
 we get
 $xa\wedge by\leqslant_Vxb\vee ay$
for $a = a_1\wedge a_2$ and for $a = a_1\vee a_2$.
\end{proof}

\begin{proposition}[see {\citealt[Satz 3]{Lor1952}}]
  The preorder \(\leqslant^{a,b}\) defines a lattice quotient of \(V\)
with left and right action of \(G\) on it such that \(b\leqslant^{a,b} a\)
if \(a\) and \(b\) are in \(G\).
\end{proposition}

\begin{proof}
We have  $b\leqslant^{a,b} a$, since $xa\wedge by\leqslant_Vxb\vee ay$ for all $x$ and $y$
by the previous Lemma.

If we have  $u\leqslant^{a,b} v$ and  $v\leqslant^{a,b} w$, then
 $xa\wedge uy\leqslant_Vxb\vee vy$ and
 $xa\wedge vy\leqslant_Vxb\vee wy$ for all $x$ and $y$.
By cut, we get $xa\wedge uy\leqslant_Vxb\vee wy$ for all $x$ and $y$,
that is  $u\leqslant^{a,b} w$. This shows that the relation $\leqslant^{a,b}$ is transitive.
This relation is also reflexive, since $xa\wedge uy\leqslant_Vuy\leqslant_Vxb\vee uy$ for all $x$ and $y$ in $G$.

 Finally, if we have $u\leqslant^{a,b} v$, that is
 $xa\wedge uy\leqslant_Vxb\vee vy$ for all $x$ and $y$ in $G$, then we also
have  $zut\leqslant^{a,b} zvt$, that is
 $xa\wedge zuty\leqslant_Vxb\vee zvty$ for all $x$ and $y$ in $G$, since we have
 $z^{-1}xa\wedge uty\leqslant_Vz^{-1}xb\vee vty$ for all $x$ and $y$ in $G$.
\end{proof}

 By definition, $u\leqslant_{a,b} v$ implies $u\leqslant^{a,b} v$, since $\leqslant_{a,b}$ is
the \emph{least} invariant preorder relation forcing $a\leqslant_{a,b} b$.

 Also by definition, note that we have $u\leqslant^{a,b} v$ iff $a\leqslant^{u,v} b$,
since $xa\wedge uy\leqslant_Vxb\vee vy$ is equivalent to $x^{-1}u\wedge ay^{-1}\leqslant_Vx^{-1}v\vee by^{-1}$.

\begin{proposition}\label{main2nc}
\(u\leqslant_{a,b} v\) and \(u\leqslant_{b,a} v\) imply
\(u\leqslant_Vv\).
\end{proposition}
\begin{proof}
In fact, $u\leqslant_{a,b} v$ implies $u\leqslant^{a,b} v$,
which implies $a\leqslant^{u,v} b$. But $u\leqslant_{b,a} v$ implies that $u$~is less than or equal to~$v$ in any lattice quotient in which $a$~is less than or equal to~$b$; therefore $u\leqslant^{u,v} v$. So $xu\wedge uy\leqslant_Vxv\vee vy$ for all $x,y$.
In particular, for $x=y=1$, we have $u\leqslant_Vv$.
\end{proof}

It follows from this that $V$ admits a group structure which extends
the one on $G$. In fact, Proposition~\ref{main2nc} reduces  the verification of the required equations to the
case where $G$ is linearly preordered by $x\vdash y$, for which $V=G$.
This is the noncommutative analogue of Theorem~\ref{thmain}.

 The difference between the noncommutative case and the commutative one is the following. 
 In the commutative
case, we give an \emph{explicit} description of the relation $\vdash_x$;  then 
we use Proposition~\ref{main3} to show that we can reason by case distinction, forcing $0\leqslant x$
or $x\leqslant 0$. In the noncommutative case, 
we use  Proposition~\ref{main2nc} to show that we can reason by case distinction, forcing
$a\leqslant b$ or $b\leqslant a$, without recourse to an explicit description of the relation $\leqslant_{a,b}$. The proof is shorter and very smart, but gives less information than in the commutative case.

\section{Examples}

Examples~\ref{exa1} and~\ref{exa2} illustrate Remark~\ref{rem}.

\begin{example} \label{exa1}
The following example is from numerical semigroups.
 
Let us consider the group $\mathbf{Z}=(\mathbb Z,0,+,-)$
preordered by the relation $x \leqslant y$
defined as $y \in x + 60 \mathbb{N}$.
We consider the meet-monoid $(S,0,+,-,\leqslant_S)$ freely generated by $\mathbf{Z}$.
The elements of $S$ are formal finite meets of elements of $\mathbf{Z}$. For example, we have in $S$
$$
a = 10 \wedge 24 \leqslant_S b = 130\wedge 84\text,
$$
since $10 \leqslant 130$ and $24 \leqslant 84$.

Now let us consider the equivariant system of ideals $T_{-7}(S)$ that
we get by forcing $0 \leqslant_{T_{-7}(S)} 7$, i.e.\ $-7\leqslant_{T_{-7}(S)} 0$ (see Proposition~\ref{LorenzenTrick}).

We have $ 3 \leqslant_{T_{-7}(S)} b  $,
since
 $$ 3 \wedge (3+7) \wedge (3+21) = 3 \wedge a \leqslant_S a \leqslant_S b\text.$$

 Yet $3  \wedge (3+21)  \not\leqslant_S b $.
 
On the other hand, we see easily that $-1\leqslant_{U_{1}(S)}0$, so that
 in the regularisation of $S$ we have $0\vdash 1$, which shows that this regularisation is the group $(\mathbb Z,0,+,-)$ with the usual linear order.
  
\end{example}

\begin{example} \label{exa2}
The following similar example is from algebraic number theory.

We consider the ring $\mathbb{Z}[x]$ with $x$ an algebraic integer solution of 
$x^3-x^2+x+7=0$. We denote by $a_1,\dots,a_k\rhd_{d}b$  the Dedekind equivariant system of ideals for the divisibility group $G$ of $\mathbb{Z}[x]$, defined as $b\in (a_1,\dots,a_k)\mathbb{Z}[x]$ for $b$ and the $a_i$'s in the fraction field $\mathbb{Q}(x)$. In fact, the finitely generated fractional ideals form a meet-monoid
$(S,\leqslant_S )$
extending the divisibility group $G$. The corresponding preorder is given
by $a_1\wedge \dots\wedge a_k\leqslant_S b_1\wedge \dots\wedge b_h$ iff each $b_i$ belongs to $(a_1,\dots,a_k)\mathbb{Z}[x]$.

The ring $\mathbb{Z}[x]$  is not integrally closed. The element $y=\frac12(x^2+1)$ of $\mathbb{Q}(x)$ is integral over $\mathbb{Z}$ and a fortiori over $\mathbb{Z}[x]$:  $y^3=y^2-4y+4$, or equivalently $1=z-4z^2+4z^3$ with $z=y^{-1}$.

Let us denote by $\vdash$ the regularisation of $S$. 
Now let us consider, for $u\in S$, the equivariant system of ideals $T_u(S)$
that we get by forcing $u\leqslant_{T_u(S)} 1$.
We see that $1\vdash y$, i.e.\  $z\vdash 1$, by showing $z\leqslant_{T_z(S)} 1$ (which holds by definition) and $z\leqslant_{T_{y}(S)} 1$, which is certified (using  Proposition~\ref{LorenzenTrick}) by  $z,z^2,z^3\leqslant_S 1$,
since the fractional ideal  $z\mathbb{Z}[x]+z^2\mathbb{Z}[x]+z^3\mathbb{Z}[x]$ contains $1$.

Yet $z,z^3\nleqslant_S 1$, as announced in Remark~\ref{rem}, since $z\mathbb{Z}[x]+z^3\mathbb{Z}[x]$ 
does not contain $1$. 
\end{example}

\begin{example} \label{exa3}
Let us consider the group $\mathbf{Z}=(\mathbb{Z},0,+,-)$
preordered by the relation $x = y$.
We compute the corresponding Lorenzen $l$-group.

We denote by $\mathbb{Z}$ the group $(\mathbb Z,0,+,-)$ with the usual order $\leqslant$, and by $\sup$ and $\inf$
the associated supremum and infimum. We denote by $\mathbb{Z}^\circ$ the conversely preordered group.

We consider the meet-monoid $(S,0,+,-,\leqslant_S)$ freely generated by $\mathbf{Z}$.
The elements of $S$ are formal finite meets of elements of $\mathbf{Z}$. We have 
$\bigwedge A \leqslant_S b$ iff $b\in A$, and $\bigwedge A \leqslant_S \bigwedge B$ iff $B\subseteq A$.

We denote by $T_n(S)$ the equivariant system of ideals that we get by forcing $n\leqslant_{T_n(S)}0$. Note that $0\leqslant_{T_{-1}(S)}b$ for $b\geqslant 0$. Using  Proposition~\ref{LorenzenTrick}, we find that $A\leqslant_{T_{-1}(S)} b$ iff $b\geqslant \inf(A)$, and $A\leqslant_{T_{1}(S)} b$ iff $b\leqslant \sup(A)$. We deduce that the regularisation of $S$ can be described as the set of intervals $\lrb{m\mkern3mu.\mkern2mu.\mkern3mun}$ inside $\mathbb{Z}$ with the order by inclusion. Equivalently, it is identified as the set of pairs $(m,n)\in \mathbb{Z}\times\mathbb{Z}^\circ$ such that $m\leqslant n$.  Now it is easy to see that the corresponding Grothendieck $l$-group is $\mathbb{Z}\times\mathbb{Z}^\circ$, where the opposite of $(m,n)$ can be identified with $(-m,-n)$. The canonical morphism $\mathbf{Z}\to \mathbb{Z}\times\mathbb{Z}^\circ$ is $m\mapsto (m,m)$. 

Note that since $\mathbf{Z}$ is the free abelian group on a singleton, we recover in this rather complicated way $\mathbb{Z}\times\mathbb{Z}^\circ$ as the free $l$-group on a singleton. 
\end{example}

\section*{Acknowledgements}

The authors thank the Hausdorff Research Institute for Mathematics
for its hospitality and for providing an excellent research environment
in May and June 2018;
part of this research has been done during its Trimester Program
\emph{Types, Sets and Constructions}.





\end{document}